\documentclass[a4paper,english,numberwithinsect]{lipicsplain-v2016}
 
\usepackage{microtype}%

\usepackage[capitalize]{cleveref}
\usepackage[utf8]{inputenc}
\usepackage[normalem]{ulem}
\usepackage{lineno}
\usepackage{mathtools}
\usepackage{epstopdf}
\usepackage{amsmath,epsfig,amsthm}
\usepackage{thmtools}
\usepackage{amssymb}
\usepackage{color}
\usepackage{tikz-cd}

 \theoremstyle{plain}
 \newtheorem{lemma}[theorem]{Lemma}
 \newtheorem{corollary}[theorem]{Corollary}
 \newtheorem{proposition}[theorem]{Proposition}

 \theoremstyle{definition}
 \newtheorem{definition}[theorem]{Definition}
 \theoremstyle{remark}
 \newtheorem{remark}[theorem]{Remark}
 \newtheorem{claim}[theorem]{Claim}
 
\usepackage[marginpar]{todo}

\makeatletter
\renewcommand{\@tododisplay}[1]{%
\marginpar{#1}%
}
\renewcommand\@displaytodo[2][\todomark]{%
\unskip%
\@tododisplay{{\todoformat #1~(\ref{todolbl:\thetodo})}}%
\footnote[\thetodo]{\todoformat #1:~#2}%
\global\@todotoks\expandafter{\the\@todotoks\todoitem{#1}{#2}}%
\@todotrue%
}%
\renewcommand\todomark{todo}
\makeatother

\newcommand{\R}{{\mathbb R}}

\newcommand{\N}{{\mathbb N}}

\newcommand{\eps}{\varepsilon}

\DeclareMathOperator{\PLReebDom}{\mathbf{PLReebDom}}

\DeclareMathOperator{\PLReebGrph}{\mathbf{PLReebGrph}}

\DeclareMathOperator{\deltaPL}{\delta_\mathit{PL}}
\DeclareMathOperator{\deltaEPL}{\delta_\mathit{ePL}}
\DeclareMathOperator{\deltaEGraph}{\delta_\mathit{eGraph}}
\DeclareMathOperator{\dFD}{\mathit{d}_\mathit{FD}}

\DeclareMathOperator{\im}{im}
\DeclareMathOperator{\openstar}{st}
\DeclareMathOperator{\carrier}{carr}

\DeclareMathOperator{\id}{id}

\title{The Reeb Graph Edit Distance is Universal\footnote{
Research supported in part by the DFG Collaborative Research Center TRR 109 \emph{Discretization in Geometry and Dynamics}. 
We thank 
Barbara Di Fabio and Yusu Wang for valuable discussions.
}
}

\author{Ulrich Bauer}
\affil{Department of Mathematics,
Technical University of Munich, Germany.
}

\author{Claudia Landi}
\affil{Dipartimento di Scienze e Metodi dell'Ingegneria, Università degli Studi di Modena e Reggio Emilia, Reggio Emilia, Italy.
}
\author{Facundo Memoli}

\affil{Department of Mathematics, The Ohio State University, Columbus, Ohio,  U.S.A. 
}

\authorrunning{U. Bauer, C. Landi, and F. M\'emoli} %

\Copyright{Ulrich Bauer, Claudia Landi, and Facundo M\'emoli}%

\EventEditors{}
\EventNoEds{2}
\EventLongTitle{34th Symposium on Computational Geometry (SoCG 2018)}
\EventShortTitle{SoCG 2018}
\EventAcronym{SoCG}
\EventYear{2018}
\EventDate{}
\EventLocation{}
\EventLogo{}
\SeriesVolume{}
\ArticleNo{}
\DOIPrefix{}

\begin{document}

\maketitle

\begin{abstract}
We consider the setting of Reeb graphs of piecewise linear functions  and study distances between them that are stable, meaning that functions which are similar in the supremum norm ought to have similar Reeb graphs. We define an edit distance for Reeb graphs and prove that it is stable and universal, meaning that it provides an upper bound to any other stable distance. In contrast, via a specific construction, we show that the interleaving distance and the functional distortion distance on Reeb graphs are not universal.
\end{abstract}

\section{Introduction}
The concept of Reeb graphs of a Morse function first appeared in \cite{Reeb46}  and was subsequently  applied to the problems in shape analysis in \cite{Sh91,HiSh*01}. The literature on Reeb graphs in the computational geometry and computational topology is ever growing (see, e.g., \cite{Bauer2014Measuring,BaMuWa15} for a discussion and references). 
The Reeb graph plays a central role in topological data analysis, not least because of the success of Mapper \cite{Mapper}, a method providing a discretization of the Reeb graph for a function defined on a point cloud.

A recent line of work has concentrated on questions about identifying suitable notions of distance between Reeb graphs: These include the so called \emph{functional distortion distance} \cite{Bauer2014Measuring}, the \emph{interleaving distance} \cite{deSilva2016Categorified}, and various \emph{graph edit distances} \cite{DiFabio2012Curves,DiFabio2016Edit,Bauer2016An}. There is of course interest in understanding the connection between different existing distances.  In this regard, it has been shown in \cite{BaMuWa15} that the functional distortion and the interleaving distances are bi-Lipschitz equivalent. 
The edit distances defined in \cite{DiFabio2012Curves,DiFabio2016Edit} for Reeb graphs of curves and surfaces, respectively, are shown to be universal in their respective setting, so the functional distortion and interleaving distances restricted to the same settings are a lower bound for those distances. Moreover, an example in \cite{DiFabio2016Edit} shows that the functional distortion distance can be strictly smaller than the edit distance considered in that paper.   

In this paper we concentrate on the setting of PL functions on compact triangulable spaces and in this realm we study the properties of \emph{stability} and \emph{universality} of distances between Reeb graphs. Inspired by a construction of distance between filtered spaces \cite{tripods}, we first construct a novel distance $\deltaPL$ based on considering joint pullbacks of two given Reeb graphs and prove it satisfies both stability and universality. Via analyzing a specific construction we then prove that neither the functional distortion nor the interleaving distances are universal. Finally, we define two edit-like additional distances between Reeb graphs that reinterpret those appearing in \cite{DiFabio2012Curves,DiFabio2016Edit,Bauer2016An} and prove that both are stable and universal. As a consequence, both distances agree with $\deltaPL$. 

\section{Topological and categorical aspects of Reeb graphs}

We start by exploring some topological ideas behind the definition of Reeb graphs. All maps and functions considered in this paper will be assumed to be continuous. Otherwise, we call them set maps and set functions.

\subsection{Reeb graphs as quotient spaces}

The classical construction of a Reeb graph \cite{Reeb46} is given via an equivalence relation as follows:

\begin{definition}\label{def-classical}
For $f:X\to \R$ a Morse function on a compact smooth manifold, the \emph{Reeb graph of $f$} is the quotient space $X/{\sim}_{f}$, with $x\sim_{f} y$ if and only if $x$ and $y$ belong to the same connected component of some level set $f^{-1}(t)$ (implying $t=f(x)=f(y)$). 
\end{definition}

While this definition was originally considered in the setting of Morse theory, it does not make explicit use of the smooth structure, and so it can be applied to a quite broad setting. However, some additional assumptions of $X$ and $f$ are justified in order to maintain some of the characteristic properties of Reeb graphs in a generalized setting.
With this motivation in mind, we revisit the definition in terms of quotient maps and functions with discrete fibers.

A \emph{quotient map} $p:X\to Y$ is a surjection such that a set $U$ is open in $Y$ if and only if $p^{-1}(U)$ is open in $X$. 
In particular, a surjection between compact Hausdorff spaces is a quotient map by the closed map lemma.
A quotient map $p:X\to Y$ is characterized by the universal property that a set map $\Phi: Y\to Z$ into any topological space~$Z$ is continuous if and only if $\Phi\circ p$ is continuous.%

The motivation for considering quotient maps and functions with discrete fibers is explained by the following fact.

\begin{proposition}\label{Reeb_quotient}
Let $f:X\to\R$ be a  function with locally connected fibers, and let $q:X\to X/{\sim}_{ f}$ be the canonical quotient map. Then the induced function $\tilde f: X/{\sim}_{ f}\to \R$ with $f=\tilde f\circ q$ has discrete fibers.
\end{proposition}

\begin{proof}
To see that the fibers of $\tilde f$ are discrete, we show that any subset $S$ of $\tilde f^{-1}(t)$ is closed. Let $T=\tilde f^{-1}(t) \setminus S$. Then $q^{-1}(T)$ is a disjoint union of connected components of $ f^{-1}(t)$. Since $\tilde f^{-1}(t)$ is locally connected, each of its connected components is open in the fiber, and so $q^{-1}(T)$ is open in $ f^{-1}(t)$, implying that $q^{-1}(S)$ is closed in $f^{-1}(t)$ and hence in $X$. Since $q$ is a quotient map, $q^{-1}(S)$ is closed if and only if $S$ is closed, yielding the claim.
\end{proof}

\subsection{Reeb quotient maps and Reeb graphs of piecewise linear functions}

We now define a class of quotient maps that leave Reeb graphs invariant up to isomorphism. The main goal is to provide a natural construction for lifting functions $f:X\to\R$ to spaces $Y$ through a quotient map $Y \to X$ in a way that yields isomorphic Reeb graphs.
To this end, we will define two categories, the category of Reeb domains and the category of Reeb graphs.

\begin{definition}\label{ReebSpCat}
We define the {\em category $\PLReebDom$ of (compact triangulable) Reeb domains} as follows:
\begin{itemize}
\item The objects of $\PLReebDom$ \emph{(Reeb domains)} are connected compact triangulable spaces. 
\item The morphisms of $\PLReebDom$ \emph{(Reeb quotient maps)} are surjective piecewise linear maps with connected fibers.
\end{itemize}
\end{definition}
The fact that this is indeed a category will be established in \cref{ReebCategory}. 

\begin{definition}\label{ReebGrCat}
The category of {\em Reeb graphs}, denoted by $\PLReebGrph$, is the category whose objects are Reeb domains $R_f$ endowed with PL functions $\tilde f:R_f\to \R$ with discrete fibers called {\em Reeb functions}, and whose morphisms between Reeb domains $R_f$ and $R_g$ respectively endowed with Reeb functions $\tilde f$ and $\tilde g$ are PL maps $\Phi:R_f\to R_g$  such that $\tilde g\circ\Phi = \tilde f$.
\end{definition}
In particular, the isomorphisms between Reeb graphs are PL homeomorphisms that preserve the function values of the associated Reeb functions.
A Reeb graph is actually a \emph{finite topological graph} (a compact triangulable space of dimension at most $1$).

\begin{theorem}\label{poly}
Any Reeb graph $R_f$ in $\PLReebGrph$ is a finite topological graph.
\end{theorem}

\begin{proof}
By definition, $\tilde f$ is (simplexwise) linear for some triangulation of $R_f$. If there were a simplex $\sigma$ of dimension at least $2$ in the triangulation of $R_f$, then  for any $x$ in the interior of~$\sigma$, the intersection $\sigma \cap \tilde f^{-1}(\tilde f(x))$ would have to be of at least dimension $1$. But this would contradict the assumption that $\tilde f$ has discrete fibers.
\end{proof}

\begin{definition}
Generalizing the classical definition (Definition \ref{def-classical}) %
, we say that a Reeb graph $R_f$ 
is \emph{the Reeb graph of} $f:X\to\R$ if 
there is a Reeb quotient map $p: X \to R_f$ such that $f=\tilde f \circ p$, where $\tilde f: R_f \to \R$ is the Reeb function of $R_f$.
\end{definition}

The following lemma shows how a transformation $g = \xi \circ f $ of a function $f$ lifts to a Reeb quotient map $\zeta$ between the corresponding Reeb graphs.

\begin{lemma}\label{psi}
Assume that $\tilde f : R_f \to \R,\tilde g : R_g \to \R$ are Reeb functions, $p_f : X \to R_f,p_g : X \to R_g$ are Reeb quotient maps, $f = \tilde f \circ p_f$, $g = \tilde g \circ p_g$, and $\xi: \im f \to \im g$ is a piecewise linear function such that $g = \xi \circ f$.
Then $\zeta = p_g \circ p_f^{-1}$ is a Reeb quotient map from $R_f$ to $R_g$.
\end{lemma}
In particular, if $\xi$ is a PL homeomorphism, then so is $\zeta$.

\begin{proof}
Let $x \in R_f$, and let $t=\tilde f(x)$. Then $C=p_f^{-1}(x)$ is a connected component of $f^{-1}(t)$ by the assumption that $p_f$ is a Reeb quotient map. 
By commutativity, we have $f^{-1} \subseteq f^{-1} \circ \xi^{-1} \circ \xi = g^{-1} \circ \xi$, and since $C$ is connected, there must be a single $y \in R_g$ with $p_g(C)=\{y\}$. Hence, $\zeta = p_g \circ p_f^{-1}$ is a set map. 
Moreover, since $p_g$ is continuous and $p_f$ is closed, the map $\zeta$ is continuous; since $p_g$ and $p_f$ are PL, the map~$\zeta$ is PL as well.

Now let $y\in R_g$ and let $s=\tilde g(y)$. Similarly to above, $C=p_g^{-1}(y)$ is a connected component of $g^{-1}(s)$. By commutativity, there is $x \in p_f(C) \subseteq R_f$ with $t = \tilde f(x) \in \xi^{-1}(s)$. Hence $\zeta$ is surjective and the fiber $\zeta^{-1}(y) = p_f(C)$ is connected.
\end{proof}

\begin{remark}
By \cref{Reeb_quotient,psi}, the Reeb graph $R_f$ of $f: X \to \R$ is isomorphic to $X/{\sim}_{f}$.
As a consequence, the Reeb graph $R_f$ together with the Reeb quotient map $p$ is unique up to a unique isomorphism, turning the Reeb graph into a universal property.
\end{remark}

We now proceed to prove that Reeb quotient maps are closed under composition. We start by showing that not only the fibers, but more generally all preimages of closed connected sets are connected.

\begin{proposition}\label{fibers}
If $p:X\to Y$ is a Reeb quotient map, then the preimage $p^{-1}(K)$ of a closed connected set $K \subseteq Y$ is connected.
\end{proposition}

\begin{proof}
Assume that $K$ is nonempty; otherwise, the claim holds trivially. Let $p^{-1}(K)=U\cup V$, with $U,V$ nonempty and closed in $p^{-1}(K)$. To show that $p^{-1}(K)$ is connected, it suffices to show that $U\cap V$ is necessarily nonempty.

Because $p^{-1}(K)$ is closed in $X$, the sets $U$ and $V$ are also closed in $X$. 
The images $p(U)$ and $p(V)$ are closed by the closed map lemma, and their union is $K$. By connectedness of $K$, their intersection is nonempty. Let $y\in p(U)\cap p(V)$.
We have
\[p^{-1}(y)=(p^{-1}(y)\cap U)\cup (p^{-1}(y)\cap V).\]
The subspaces $(p^{-1}(y)\cap U)$ and $(p^{-1}(y)\cap V)$ are closed in $p^{-1}(y)$,
and by connectedness of the fiber $p^{-1}(y)$, their intersection must be nonempty. In particular, $U\cap V$ is nonempty.
\end{proof}

\begin{corollary}\label{composition}
If $p:X\to Y$ and $q:Y\to Z$ are Reeb quotient maps, then the composition $q\circ p:X\to Z$ is a Reeb quotient map too.
\end{corollary}

As mentioned before, the main purpose of Reeb quotient maps is to lift Reeb functions to larger domains while maintaining the same Reeb graph. The following property is a consequence of the above statement:

\begin{corollary}\label{homeoR}
Let $f: X \to \R$ be a function with Reeb graph $R_f$, %
and let $q: Y \to X$ be a Reeb quotient map. Then $R_f$ %
is also the Reeb graph of 
$f\circ q: Y \to \R$.

\end{corollary}

\begin{proof}
Let $\tilde f: R_f \to \R$ be the Reeb function of $R_f$ and $p: X \to R_f$ be the Reeb quotient map factoring $f = \tilde f \circ p$. Then by \cref{composition}, $R_f$ is also a Reeb graph for $f \circ q = \tilde f \circ (p \circ q): Y \to \R$ via the Reeb quotient map $p \circ q: Y \to R_f$.
\end{proof}

We now show that Reeb quotient maps are stable under pullbacks.

\begin{proposition}\label{pullback}
Consider the pullback diagram
\[
\begin{tikzcd}[column sep={12.5mm,between origins},row sep=5.5mm]
 & Y & \\
 X_1\ar[ur,"p_1"]& & X_2\ar[ul,"p_2",swap] \\
& X_1\times_{Y}X_2\ar[ul,"q_1"]\ar[ur,"q_2",swap]
\end{tikzcd}
\]
If 
the map $p_1$ (resp. $p_2$) is a Reeb quotient map, then so is the map $q_2$ (resp. $q_1$).
\end{proposition}

\begin{proof}

First note that the category of compact triangulable spaces has all pullbacks \cite{Stallings93}.
For $x_2\in X_2$, by surjectivity of $p_1$ there is some $x_1\in X_1$ such that $p_1(x_1)=p_2(x_2)$. Thus $(x_1,x_2)\in X_1\times_Y X_2$ and $q_2(x_1,x_2)=x_2$, proving that $q_2$ is surjective. 
Moreover, 
for $x_2\in X_2$, we have $q_2^{-1}(x_2)=p_1^{-1}(p_2(x_2))\times\{x_2\}$. By assumption, $p_1^{-1}(p_2(x_2))$ is connected being a fiber of $p_1$, implying that $p_1^{-1}(p_2(x_2))\times\{x_2\}$ is connected. 
Finally, applying \cref{fibers} to $q_2$, we obtain that the pullback space $X_1\times_{Y}X_2$ is connected. The proof for $q_1$ is analogous.
\end{proof}

\begin{theorem}\label{ReebCategory}
The Reeb domains and Reeb quotient maps form a finitely complete category, i.e., every finite diagram has a limit.
\end{theorem}
\begin{proof}
By \cref{composition}, the Reeb quotient maps are closed under composition and contain the identity maps of Reeb domains, so they form a category. This category has all pullbacks by \cref{pullback}, and the one-point space is a terminal object, so equivalently it has all finite limits \cite[Prop.~5.14~and~5.21]{Awodey}.
\end{proof}

\section{Stable and universal distances}
Throughout this paper, we will use the term \emph{distance} to describe an extended pseudo-metric $d: X\times X \to [0,\infty]$ on some collection $X$. Our main goal is the introduction of a distance between Reeb
graphs that is stable and universal in the following sense. 

\begin{definition}\label{defSIU}
We say that a distance $d_S$ on the objects of $\PLReebGrph$ is \emph{stable} if and only if given any two Reeb graphs $R_f$ and $R_g$ respectively endowed with Reeb functions $\tilde f$ and $\tilde g$, for any Reeb domain $X$ with Reeb quotient maps $p_f: X \to R_f$ and $p_g: X \to R_g$ we have
\begin{equation}
d_S(R_f,R_g)\le \|\tilde f\circ p_f-\tilde g\circ p_g\|_\infty. \tag{S}\label{stability}
\end{equation}
Note that stability implies that isomorphic Reeb graphs have distance $0$. Indeed, an isomorphism of Reeb graphs $\gamma:R_f\to R_g$ yields $d_S(R_f,R_g) \leq \|\tilde f\circ \id - \tilde g\circ \gamma\|_\infty=0$. 

Moreover, we say that a stable distance $d_U$ on the objects of $\PLReebGrph$ is {\em universal} if and only if for any other stable distance $d_S$ on $\PLReebGrph$, we have
\begin{equation}
d_S(R_f,R_g)\le d_U(R_f,R_g). \tag{U}\label{universality}
\end{equation}
\end{definition}

\begin{remark}\label{RemDef}
By connectedness of $R_f$ and $R_g$, there is at least one space $X$ with maps $p_f,p_g$ as needed to define the stability property: $X=R_f\times R_g$, with $p_f,p_g$ the canonical projections. The resulting functions $f = \tilde f\circ p_f , g = \tilde f\circ p_f : R_f\times R_g\rightarrow \R$ then satisfy $\|f-g\|_\infty = \max(\sup f, \sup g) - \min (\inf f, \inf g)$. In particular, for compact Reeb graphs a stable distance is always finite.
\end{remark}

The definition of stability yields the following canonical universal distance.

\begin{definition}\label{def_canonical}
For any two Reeb graphs $R_f$ and $R_g$ endowed with  Reeb graph functions  $\tilde f$ and $\tilde g$, let
\[\deltaPL(R_f,R_g):=\inf_{p_f:R_f\leftarrow X\rightarrow R_g: p_g} \|\tilde f\circ p_f-\tilde g\circ p_g\|_\infty,\]
where $X$ is any Reeb domain, and $p_f,p_g$ are Reeb quotient maps. 
\end{definition}

\begin{proposition}\label{canonical}
The distance $\deltaPL$ is the largest stable distance on $\PLReebGrph$. Hence, $\deltaPL$ is universal.
\end{proposition}

\begin{proof}
To see that $\deltaPL$ is a distance, the only non-trivial part is showing the triangle inequality. To this end, given diagrams $p_f:R_f\leftarrow X\rightarrow R_g: p_g$ and $p'_g:R_g\leftarrow Y\rightarrow R_h: p_h$, we can pullback the diagram $p_g:X\rightarrow R_g\leftarrow Y:p'_g$ to obtain the diagram $q_X:X\leftarrow X\times_{R_g}Y\rightarrow Y: q_Y$, where $X\times_{R_g}Y$ is a Reeb domain and $q_X,q_Y$ are Reeb quotient maps by \cref{pullback}. 
Defining $f = \tilde f\circ p_f\circ q_X$, $g = \tilde g\circ p_g\circ q_X=\tilde g\circ p'_g\circ q_Y$, and $h = \tilde h\circ p_h\circ q_Y$, 
we have
\begin{align*}
\deltaPL(R_f,R_h) \leq \|f-h\|_\infty &\le \|f-g\|_\infty + \|g-h\|_\infty 
\\
&\leq
\|\tilde f\circ p_f - \tilde g\circ p_g\|_\infty + \|\tilde g\circ p'_g - \tilde h\circ p_h\|_\infty,
\end{align*}
where the last inequality holds because $\im q_X\subseteq X$ and $\im q_Y\subseteq Y$.
Hence, $\deltaPL(R_f,R_h)\le \deltaPL(R_f,R_g)+\deltaPL(R_g,R_h)$.
By definition of stability, $d_S\le \deltaPL$ for any stable distance $d_S$ defined on the objects of $\PLReebGrph$, implying that $\deltaPL$ is universal. 
\end{proof}

\begin{example}
Consider the one point Reeb graph $\ast_c$ endowed with the function identical to $c\in\mathbb{R}.$ Then, for any Reeb graph $R_f$ endowed with the function $\tilde{f}$, $\deltaPL(R_f,\ast_c) = \|\tilde{f}-c\|_\infty.$
\end{example}

\label{cylinder}

We now consider an example where we can explicitly determine the value of the distance $\deltaPL(R_f,R_g)$ between two specific simple Reeb graphs $R_{f} =  \mathbb{S}^1=\{(x,y)\in\R^2: x^2+y^2=1\}$ with $\tilde f(x,y)=x$ and $R_{g} = [-1,1]$ with $\tilde g(t)=t$. The example demonstrates the non-universality of certain distances proposed in the literature. 
We prove:
\begin{proposition} $\deltaPL(R_f,R_g)=1$.
\end{proposition}
The proof of this proposition will be obtained from the two claims below.
\begin{claim}
$\deltaPL(R_{f},R_{g}) \le 1$. 
\end{claim}
\begin{proof}
Consider the cylinder $C=\{(x,y,z)\in\R^3: x^2+y^2=1, \, |2z-x| \le 1\}$ together with functions $f(x,y,z)=x$ and $g(x,y,z)=z$ defined on $C$.
\begin{figure}[h]
\begin{center}
\includegraphics[width=1in]{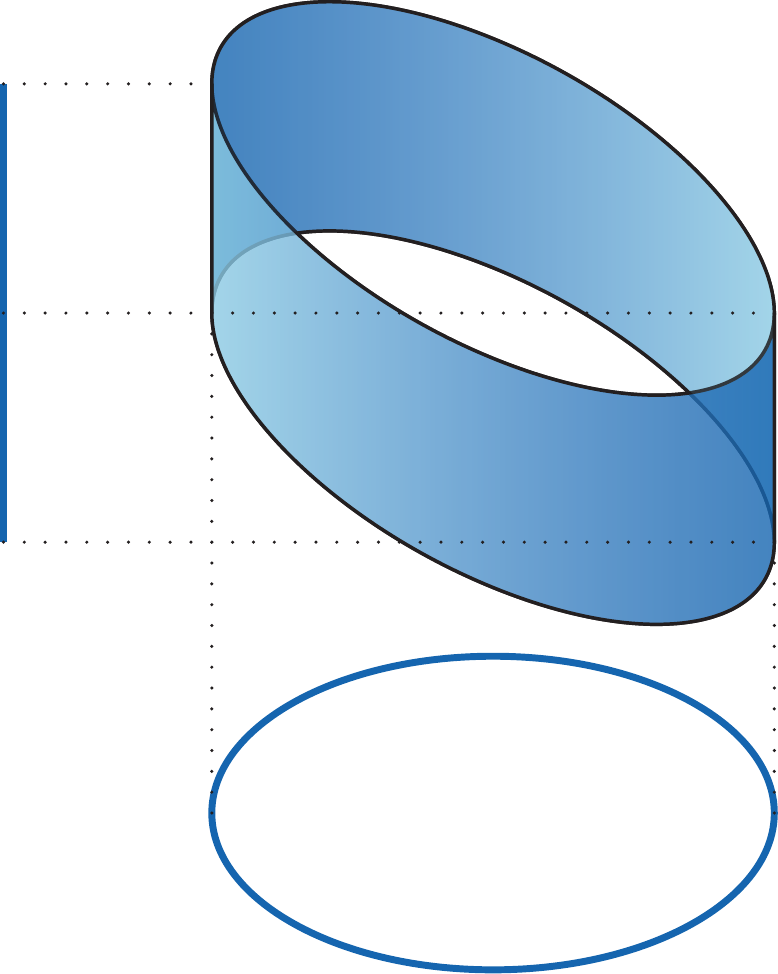}%
\vspace{-12pt}
\end{center}
\end{figure}
Then $R_f$ is a Reeb graph of $f$ via the Reeb quotient map $(x,y,z) \mapsto (x,y)$, and 
$R_g$ is a Reeb graph of $g$ via the Reeb quotient map $(x,y,z) \mapsto z$.
 Since we  have $|f(c)-g(c)| \leq 1$ for all $c\in C$, this implies that $\deltaPL(R_{f},R_{g}) \le 1$. 
\end{proof}
\begin{claim}\label{claim:lb}
$\deltaPL(R_{f},R_{g}) \ge 1$. 
\end{claim}
\begin{proof}
Assume for a contradiction that there is a diagram $p_f:R_f\leftarrow Z\to R_g: p_g$ of Reeb quotient maps such that,
letting
$\hat f = \tilde f \circ p_f$ and $\hat g = \tilde g \circ p_g$,
we have $\| \hat f -\hat g \|_\infty = \delta < 1.$ We then observe the following:
\begin{itemize}
\item $\hat g^{-1}(0) \subseteq \hat f^{-1}([-\delta,+\delta])$.
\item $\tilde f^{-1}([-\delta,+\delta])$ consists of two circular arcs homeomorphic by $\tilde f$ to $[-\delta,+\delta]$, and thus, by \cref{fibers}, $\hat f^{-1}([-\delta,+\delta])%
$ consists of two connected components $C_+$ and $C_-$ as well. 
\item For both components we have $\hat f(C_\pm) = [-\delta,\delta]$, and so $\|\hat f -\hat g \|_\infty = \delta$ implies that $0 \in \hat g(C_\pm)$.
Thus $\hat g^{-1}(0)\cap C_-\neq \emptyset$ and  $\hat g^{-1}(0)\cap C_+\neq \emptyset$. 
\end{itemize}
But since $\hat g^{-1}(0)\subseteq 
C_-\sqcup C_+$, this would contradict the assumption that the fiber $\hat g^{-1}(0)$ is connected.
\end{proof}

The current example illustrates that the \emph{functional distortion distance} introduced in \cite{Bauer2014Measuring} and the \emph{interleaving distance} introduced in \cite{deSilva2016Categorified} both fail to be universal. We first recall the definition of the former. For any Reeb graph $R_f$ with Reeb function $\tilde f$, consider the metric on $R_f$ given by 
\[d_f(x,y)=\inf\{b-a \mid x,y \text{ are in the same connected component of } \tilde f^{-1}([a,b])\}.\]
Given maps $\phi: R_f \to R_g$ and $\psi: R_g \to R_f$, we write
\[
G(\phi,\psi)=\big\{(p,\phi(p)):p\in R_f\}\cup\{(\psi(q),q):q\in R_g\big\} \]
for the correspondences induced by the two maps.
The functional distortion distance is
\[
\dFD(R_f, R_g) = \inf_%
{\phi,\psi}(
\max \big\{ \sup_{(p,q),(p',q')\in G(\phi,\psi)}\frac12\left|d_f(p,p')-d_g(q,q')\right|, \|f-g\circ\phi\|_\infty, \|f\circ\psi-g\|_\infty \big\}).  
\]

To see that neither the functional distortion distance nor the interleaving distance are universal we establish:
\begin{proposition}
$d_{I}(R_f,R_g) 
\leq \dFD(R_f,R_g) \leq \frac12.$
\end{proposition}
\begin{proof}
 By \cite[Lemma~8]{BaMuWa15}, the functional distortion distance is an upper bound on the interleaving distance on Reeb graphs \cite{deSilva2016Categorified}, and so it is enough to prove that $\dFD(R_f,R_g) \leq \frac12.$ To this end consider the maps $\phi:R_f\to R_g, ~ (x,y) \mapsto x \quad \text{and} \quad \psi:R_g\to R_f, ~ t \mapsto \left(t, \sqrt{1-t^2}\right).$
For every pair $p, p' \in R_f$ one can verify that $ |\tilde f(p)-\tilde f(p')|\leq d_f(p, p') \leq |\tilde f(p)- f(p')|+1$, while for every pair $q,  q' \in R_g$, we have $d_g(q, q') = |\tilde g(q)-\tilde g(q')|$. This implies that for any two corresponding pairs $(p,q),(p',q')\in G(\phi,\psi)$, we have $|d_f(p, p')-d_g(q,q')|\leq 1$, and thus $D(\phi,\psi) \leq \frac12$. Moreover, both maps preserve function values, so $\dFD(R_f,R_g) \leq \frac12$. 
\end{proof}

\section{The topological and graph edit distances}

Given a pair of Reeb graphs $R_f, R_g$, consider a diagram of the form
\begin{equation}\label{zig-zag}
\begin{tikzcd}[column sep={12.5mm,between origins},row sep=5.5mm]
\R && \R && && \R && \R\\
R_f = R_1\ar[u,"\tilde f_1"] && R_2\ar[u,"\tilde f_2"] &&\cdots&& R_{n-1}\ar[u,"\tilde f_{n-1}"]&&R_n = R_g\ar[u,"\tilde f_n"]\\
&X_1 \ar[ul,two heads] \ar[ur,two heads] &&X_2\ar[ul,two heads]\ar[ur,dotted]&&X_{n-2}\ar[ul,dotted]\ar[ur,two heads]&&X_{n-1}\ar[ul,two heads]\ar[ur,two heads]
\end{tikzcd}
\end{equation}
where for $n\in\N$ $\tilde f_1,\ldots,\tilde f_n$ are Reeb functions with $\tilde f_1=\tilde f$ and $\tilde f_n=\tilde g$, and the maps
$X_i \to R_i, R_{i+1}$ for $i=1,\ldots,n-1$, are Reeb quotient maps.
We call the diagram a \emph{Reeb zigzag diagram} between $R_f$ and $R_g$. Observe that, by \cref{RemDef}, between any two Reeb graphs $R_f$ and $R_g$ there exists a Reeb zigzag diagram.

A Reeb zigzag diagram can be regarded as being composed of the following elementary diagrams:
\begin{equation*}
\begin{tikzcd}[column sep={12.5mm,between origins},row sep=5.5mm]
& \R \\
& R_{i}\ar[u,"\tilde f_{i}"]\\
X_{i-1}\ar[ur,two heads]&&X_{i}\ar[ul,two heads]
\end{tikzcd}
\qquad
\qquad
\begin{tikzcd}[column sep={12.5mm,between origins},row sep=5.5mm]
 \R && \R\\
R_{i}\ar[u,"\tilde f_{i}"]&&R_{i+1}\ar[u,"\tilde f_{i+1}"]\\
&X_{i}\ar[ul,two heads]\ar[ur,two heads]
\end{tikzcd}
\end{equation*}
This way, we may think of a Reeb zigzag diagram as a sequence of operations transforming the $R_f$ into $R_g$. The elementary diagram on the left corresponds to an \emph{edit} operation: the space $X_{i-1}$, together with a function $X_{i-1}\to \R$ with Reeb graph $R_i$, is transformed to another space $X_{i}$, with a function $X_{i}\to \R$ having the same Reeb graph $R_i$.
The elementary diagram on the right corresponds to a \emph{relabel} operation: the function on $X_i$ with Reeb graph $R_i$ is transformed to another function with Reeb graph $R_{i+1}$. The idea of edit and relabel operations is inspired by previous work on edit distances for Reeb graphs \cite{DiFabio2016Edit,Bauer2016An}.

In order to define an edit distance using Reeb zigzag diagrams, we need to
assign a cost to a given Reeb zigzag diagram between $R_f$ and $R_g$. To that end, we can consider a cone from a space~$V$ by Reeb quotient maps $V \to R_i$:
\begin{equation}\label{cone}
\begin{tikzcd}[column sep={12.5mm,between origins},row sep=5.5mm]
\R && \R && && \R && \R\\
R_1\ar[u,"\tilde f_1"] && R_2\ar[u,"\tilde f_2"] &&\cdots&& R_{n-1}\ar[u,"\tilde f_{n-1}"]&&R_n\ar[u,"\tilde f_n"]\\
&X_1 \ar[ul,two heads] \ar[ur,two heads] &&X_2\ar[ul,two heads]\ar[ur,two heads,dotted]&\cdots&X_{n-2}\ar[ul,two heads,dotted]\ar[ur,two heads]&&X_{n-1}\ar[ul,two heads]\ar[ur,two heads]\\
&&&&V
\ar[ulll,two heads]
\ar[ul,two heads]
\ar[ur,two heads]
\ar[urrr,two heads]
&&&&\
\end{tikzcd}
\end{equation}
We call this diagram a \emph{Reeb cone}. Any Reeb zigzag diagram admits such a cone. Indeed, the category $\PLReebDom$ has all finite limits by \cref{ReebCategory}, and the limit over the lower part of diagram \eqref{zig-zag}, consisting of Reeb quotient maps, yields a limit over the whole diagram.
In a Reeb cone, by commutativity, each of the Reeb functions $\tilde f_i$ induces a unique function $f_i:V \to \R$. By \cref{homeoR}, the Reeb graph of $f_i$ is isomorphic to $R_i$.
This way, we pull back the individual functions $\tilde f_i$ to functions $f_i$  on a common space with the same Reeb graphs, where they can be compared using the supremum norm.

Using these ideas, we can now introduce distances on the objects of $\PLReebGrph$, and proceed to prove that they are stable and universal. 

\begin{definition}\label{costs}
Given a Reeb cone from a space $V$ as in \eqref{cone}, we define the \emph{spread} of the functions
$(f_i)_{i=1,\ldots,n}:V \to \R$, as the function $s^V:V\to\R$, $x \mapsto \max_{i=1,\ldots,n}f_i(x) - \min_{j=1,\ldots,n}f_j(x)$. 
Moreover, for a Reeb zigzag diagram $Z$ between $R_f$ and $R_g$ as in \eqref{zig-zag}, consider the limit of $Z$, denoted by $L$.
The \emph{cost} of the Reeb zigzag diagram $Z$ is the
supremum norm of the spread $s^L$,
\[
c_Z := \|s^L\|_\infty = \sup_{x\in L}\left(\max_i f_i(x) -\min_j f_j(x)\right).
\]
\end{definition}

\begin{definition}\label{distances}
We define the \emph{(PL) edit distance} $\deltaEPL$ between Reeb graphs $R_f$ and $R_g$ in $\PLReebGrph$ as the infimum cost of all
Reeb zigzag diagrams $Z$ in $\PLReebDom$ between $R_f$ and $R_g$:
\[\deltaEPL(R_f,R_g)=\inf_{Z} c_{Z}.\]
Moreover, we define the 
\emph{graph edit distance} $\deltaEGraph$ between Reeb graphs $R_f$ and $R_g$ in $\PLReebGrph$ analogously by restricting the infimum to Reeb zigzag diagrams $Z$ where all the
spaces $X_i$ and $R_i$ are finite topological graphs, and all the maps are PL.
\end{definition}
Thus, on $\PLReebGrph$ we have two edit distances, satisfying
\begin{equation}\label{eT<eG}
\deltaEPL\le \deltaEGraph.
 \end{equation}
The Reeb graph edit distance $\deltaEGraph$ is a categorical reformulation of the definition given in \cite{Bauer2016An}. The main goal is to prove that these distances have the
stability and universality properties (\cref{stabilityPL,univPL,stabilityGraph,univGraph}). As a consequence, whenever
applicable, they actually coincide with the canonical universal distance $\deltaPL$ defined in \cref{def_canonical}:

\begin{corollary}
\(\deltaPL=\deltaEPL= \deltaEGraph.\)
\end{corollary}
The proofs of stability and universality for $\deltaEPL$ are straightforward and are given next. The verification of stability and universality for $\deltaEGraph$ follows in Section \ref{sec:rged}.

\begin{proposition}\label{stabilityPL}
$\deltaEPL$ is a stable distance.
\end{proposition}

\begin{proof}
Let $R_f, R_g$ be Reeb graphs with Reeb functions $\tilde f$ and $\tilde g$. For any space $X$ such that there exist two Reeb quotient maps $p_f:X\to R_f$ and $p_g:X\to R_g$, the diagram
\[
\begin{tikzcd}[column sep={12.5mm,between origins},row sep=5.5mm]
\R && \R \\
R_f \ar[u,"\tilde f"] && R_g\ar[u,"\tilde g",swap] \\
&X \ar[ul,two heads,"p_f"] \ar[uul,"f",swap] \ar[ur,two heads,swap,"p_g"] \ar[uur,"g"]
\end{tikzcd}
\]
is a Reeb zigzag diagram with limit object $X$.
The cost of this Reeb zigzag diagram is
exactly $\|f-g\|_\infty$. Hence, $\deltaEPL(R_{f},R_{g})\le
\|f-g\|_\infty$.
\end{proof}

Our proof of universality of the edit distance is similar to previous universality proofs for the bottleneck distance \cite{dAmico} and for the interleaving distance \cite{Lesnick}.

\begin{proposition}\label{univPL}
$\deltaEPL$ is a universal distance.
\end{proposition}

\begin{proof}
Let $R_f, R_g$ be Reeb graphs with Reeb functions $\tilde f$ and $\tilde g$. Let $\deltaEPL(R_f,R_g)=d$. Hence,
for any $\eps>0$, there is a Reeb zigzag diagram $Z$ between $R_f=R_1$ and $R_g=R_n$, with limit $L$ and functions $f_i$ as in \cref{costs}, having cost \[c_{Z} = \|s^L\|_\infty = \|\max_i f_i-\min_j f_j\|_\infty \le d+\eps.\]
Let $p_f: L \to R_f$ and $p_g: L \to R_g$ be the induced Reeb quotient maps.
If
$d_S$ is any other stable distance (cf. Definition \ref{defSIU}) between $R_f$ and $R_g$, we have
\[d_S(R_f,R_g) \leq \|\tilde f\circ p_f-\tilde g\circ p_g\|_\infty
\leq
\|\max_i f_i-\min_j f_j\|_\infty \le d+\eps
.\]
Since the above holds for all $\eps > 0$, we have
$d_S(R_f,R_g) \le d = \deltaEPL(R_f,R_g).$
\end{proof}

\section{Stability and universality of the Reeb graph edit distance}\label{sec:rged}

We now turn to the proof of stability and universality for the Reeb graph edit distance.
Recall that, in the case of $\deltaEGraph$, the admissible Reeb zigzag diagrams are PL zigzags of finite topological graphs.  As mentioned above, the distance $\deltaEGraph$ is applicable to Reeb graphs of compact triangulable spaces. 

\begin{lemma}\label{phi}
Let $X=|K|$ and let $V$ be the vertex set of $K$. Let $f,g:X\to \R$ be PL functions, simplexwise linear on $K$. Let $\chi:\im f\to \im g$ be a weakly order preserving PL surjection such that $\chi \circ f(v)=g(v)$ for every vertex $v\in V$. Then there is a Reeb quotient map $X/{\sim}_f \to X/{\sim}_g$.
\end{lemma}

\begin{proof}
For simplicity, we write $R_f=X/{\sim}_f$, $R_g=X/{\sim}_g$, and $R_h=X/{\sim}_h$, where $h=\chi\circ f$.
Applying \cref{Reeb_quotient}, $f$ can be factorized as $f=\tilde f\circ q_f$, where $q_f: X\to R_f$ is the canonical projection and $\tilde f: R_f \to \R$ is a Reeb function. Analogously, we obtain $g=\tilde g\circ q_g$ and $h=\tilde h\circ q_h$.
We show that there is a Reeb quotient map $k:X\to R_h$ making the following diagram commute:
\[
\begin{tikzcd}[column sep={12.5mm,between origins},row sep=5.5mm]
\im f \ar[rr, two heads, "\chi"] && \im g \\
R_f \ar[u,"\tilde f"] & R_h \ar[ur,swap,"\tilde h"] & R_g \ar[u,"\tilde g",swap] \\
X \ar[u,two heads,"q_f"] \ar[ur,two heads,swap,"q_h"] && X\ar[ul,two heads,dotted,"k"] \ar[u,two heads,swap,"q_g"]
\end{tikzcd}
\]
The claim then follows by applying \cref{psi} to obtain Reeb quotient maps $R_f \to R_h$ and $R_h \to R_g$, which compose to the desired map $R_f \to R_g$.

In order to prove the existence of such a Reeb quotient map $k$, we define the relation 
\[k=q_h\circ((h^{-1}\circ g) \cap \openstar_K)\] on $X\times R_h$. 
Here $\openstar_K$ denotes the open star on $X = |K|$, defined as
\[\openstar_K(x) = \{ y \in X \mid \sigma \in K, y \in \sigma^\circ, x \in \sigma\}.\]
Note that the converse relation to the open star is the \emph{(closed) carrier}, $\openstar_K^{-1}=\carrier_K$, where $\carrier_K(A)$ is the underlying space of the smallest subcomplex of $K$ containing $A \subseteq X$. We will also use the \emph{open carrier} relation $\carrier_K^\circ$, where $\carrier_K^\circ(A)$ is the smallest union of open simplices of $K$ covering $A$. Note that the open carrier relation is symmetric, i.e., $(\carrier_K^\circ)^{-1}=\carrier_K^\circ$. Moreover, we have $\carrier_K^\circ \subseteq \openstar_K$.

The remainder of the proof is split into several lemmas. \Cref{h-g} describes the behaviour of the functions $h$ and $g$ on the simplices of $K$. \Cref{k-continuous-surjection} shows that $k$ is a continuous surjection, and \cref{k-connected-fibers} shows that $k$ has connected fibers.  Since $\tilde{h}\circ k = g$, we conclude that $k$ is PL. 
Thus, $k$ is a Reeb quotient map, and the claim follows from \cref{psi}.
\end{proof}

\begin{lemma}\label{h-g}
For every simplex $\sigma$ in $K$, $g(\sigma)=h(\sigma)$ and $g(\sigma^\circ) \subseteq h(\sigma^\circ)$.
\end{lemma}

\begin{proof}
We have $h(\sigma)=g(\sigma)$ because $h$ is equal to $g$ on the vertices of $K$, and $h=\chi\circ f$ with $f$ linear on $\sigma$ and $\chi$ a weakly order preserving surjection.

To show that $g(\sigma^\circ) \subseteq h(\sigma^\circ)$, note that since $g$ is linear on $\sigma$, either $g$ is constant on $\sigma$ and so $g(\sigma^\circ) = g(\sigma) = h(\sigma)$, or $g(\sigma^\circ) = (g(v),g(w))$ for some vertices $v,w$ of $\sigma$.
In the latter case, since $h$ and $g$ coincide on the vertices,
we have $g(\sigma^\circ) = g(\sigma)^\circ = h(\sigma)^\circ$.
Finally, since $h(\sigma^\circ) \subseteq h(\sigma) \subseteq \overline{h(\sigma^\circ)}$ are nested intervals, we have 
$h(\sigma)^\circ \subseteq h(\sigma^\circ)$
and the claim follows.
\end{proof}

\begin{lemma}
\label{k-continuous-surjection}
$k$ is a continuous surjection.
\end{lemma}

\begin{proof}
We fist show that $k$ is right-unique, i.e., for any $x \in X$ and $y,y' \in k(x)$, we have $y=y'$.
To see this, let $t=g(x)$ and note that $\tilde h(y)=\tilde h(y')=t$.
Let $\sigma\in K$ be such that $x \in \sigma^\circ$.
By \cref{h-g} there is a point $\zeta \in \sigma^\circ$ with $h(\zeta)=g(x)=t$; in particular, $\zeta \in h^{-1}(t) \cap \openstar_K(x)$.
Furthermore, there are points $\xi,\xi'\in h^{-1}(t) \cap \openstar_K(x)$ with $\xi \in q_h^{-1}(y)$ and $\xi'\in q_h^{-1}(y')$. 
But since $h^{-1}(t) \cap \tau$ is necessarily connected for every simplex $\tau$, 
we know that $\zeta$ lies in the same connected component of $h^{-1}(t) \cap \openstar_K(x)$ as both $\xi$ and $\xi'$, and so
we have $y=q_h(\xi)=q_h(\xi')=y'$
as claimed.

To show that $k$ is left-total, 
we need to show that for every $x\in X$, $k(x)\ne \emptyset$. It suffices to show that 
for every $x\in X$, 
$\openstar_K(x)$ contains a point $x'$ with $h(x')=g(x)$. 
This follows by considering the simplex $\sigma\in K$ with $x \in \sigma^\circ$. Now by \cref{h-g}, there is a point $x' \in \sigma^\circ \subseteq \openstar_K(x)$ with $h(x')=g(x)$ as claimed.

To show that $k$ is right-total, we show that for every $y \in R_h$, there is some \[x\in k^{-1}(y) = ({\carrier_K} \circ q_h^{-1})(y)\cap(g^{-1}\circ \tilde h)(y),\]
or equivalently, there is some $x\in {\carrier_K} \circ q_h^{-1}(y)$ such that $g(x)= \tilde h(y)$.
If $q_h^{-1}(y)$ contains some vertex $v$ of $K$, choose $x=v$. Otherwise,
let $\xi\in q_h^{-1}(y)$, and let $\sigma\in K$ be such that $\xi \in \sigma^\circ$. 
Now by \cref{h-g} there is a point $x \in \sigma \subseteq {\carrier_K} \circ q_h^{-1}(y)$ with $g(x) = h(\xi) = \tilde h(y)$.

Finally, to show that $k$ is continuous, we show that for every closed subset $L$ of $R_h$, the preimage $k^{-1}(L)$ is closed.
Since $k^{-1}=({\carrier_K}\circ q_h^{-1})\cap (g^{-1}\circ \tilde h)$, it is sufficient to show that both ${\carrier_K}\circ q_h^{-1}(L)$ and $g^{-1}\circ \tilde h(L)$ are closed in $X$.
First note that ${\carrier_K}\circ q_h^{-1}(L)$ is closed as a subcomplex of $K$.
Furthermore, the image $\tilde h(L)$ is closed by the closed map lemma. By continuity of $g$ it follows that $g^{-1}\circ \tilde h(L)$ is closed in $X$.
\end{proof}

\begin{lemma}
\label{k-connected-fibers}
The fibers of $k$ are connected.
\end{lemma}

\begin{proof}
Let $y \in R_h$ be a point in the Reeb graph with value $t = \tilde h(y)$, and $C = q_h^{-1}(y) \subseteq h^{-1}(t)$ the corresponding component of the level set of $h$.
Let $U=\carrier_K (C)$, and let $L$ be the corresponding subcomplex of $K$. Writing $D=k^{-1}(y)$, we have $C = U \cap h^{-1}(t)$ and $D = U \cap g^{-1}(t)$.
To prove that $D$ is connected, it is sufficient to show that $C$ and $D$ have finite closed covers with isomorphic nerves; since $C$ is connected, both nerves and hence also $D$ are then connected too.

The cover of $C$ is given by $\{\sigma \cap C \mid \sigma \in L\}$, and similarly the cover of $D$ is $\{\sigma \cap D \mid \sigma \in L\}$.
Observe that any two cover elements of $C$, say $\sigma \cap C$ and $\tau \cap C$, have a nonempty intersection $(\sigma \cap C)\cap (\tau \cap C) = (\sigma \cap \tau) \cap C$ if and only if $t \in h(\sigma \cap \tau)$. Similarly, $\sigma \cap D$ and $\tau \cap D$ have nonempty intersection if and only if $t \in g(\sigma \cap \tau)$. But $g(\sigma \cap \tau)=h(\sigma \cap \tau)$ by \cref{h-g}, and so the nerves of both covers are isomorphic as claimed.
\end{proof}

We thus have shown the existence of the Reeb quotient map $k$. This completes the proof of \cref{phi}. We will now apply \cref{phi} to construct Reeb graph edit zigzags from straight line homotopies.

\begin{lemma}\label{lambda_i}
Let $X=|K|$ be a compact triangulable space, with PL functions $f, g : X\to \R$, simplexwise linear on $K$. Consider the straight line homotopy ${f}_\lambda=(1-\lambda) f+\lambda g$, with $0\le \lambda\le 1$. Then there exists a partition
$0=\lambda_1<\cdots <\lambda_n= 1$ such that
for every $1 \leq i < n$ and $\rho\in (\lambda_i,\lambda_{i+1})$, there exist weakly order preserving PL surjections $\chi_i:\im {f}_{\rho} \to \im f_{\lambda_i}$ and $\xi_{i+1}:\im {f}_{\rho} \to \im f_{\lambda_{i+1}}$ with \[\chi_i\circ{f}_{\rho}(v)=f_{\lambda_i}(v) ~~ \text{and} ~~ \xi_{i+1}\circ{ f}_{\rho}(v)=f_{\lambda_{i+1}}(v)\] for every vertex $v$ in $K$.
\end{lemma}

\begin{proof}
Consider the set of values $0 < \lambda < 1$ such that there exist vertices $v,w\in K$ with
\begin{equation*}\label{lambda}
{ f}_{\lambda}(v)={f}_{\lambda}(w), ~~\text{but}~~ { f}_{\rho}(v)\neq{ f}_{\rho}(w) ~~ \text{for every}~~ \rho\neq\lambda.
\end{equation*}
This set is finite because the function $\lambda \mapsto f_\lambda(v)-f_\lambda(w)$ is linear and $K$ has a finite number of vertices. Let $\{\lambda_i\}_{1\leq i\leq n}$ be this set together with $0$ and $1$, indexed in ascending order. By the linearity of ${ f}_\lambda$ with respect to the parameter $\lambda$, we also see that the order induced by $f_\rho$ on the vertices is the same for every $\rho\in (\lambda_i,\lambda_{i+1})$. Indeed, if there exist two distinct vertices $v,w$ of $K$ such that ${f}_{\rho}(v)={f}_{\rho}(w)$ for some $\rho\in(\lambda_i,\lambda_{i+1})$, then ${ f}_{\lambda}(v)={ f}_{\lambda}(w)$ for every $\lambda\in[0,1]$.
By continuity, the order is still weakly preserved along $[\lambda_i,\lambda_{i+1}]$.

Therefore, the function $f_\rho(v) \mapsto { f}_{\lambda_i}(v)$  is well-defined and can be extended to a 
piecewise linear function $\chi_i$ satisfying the claim. The function $\xi_{i+1}$ can be defined similarly.
\end{proof}

\begin{theorem}\label{stabilityGraph}
$\deltaEGraph$ is a stable distance.
\end{theorem}

\begin{proof}
Let $X=|K|$ be a compact triangulable space and $f, g : X\to \R$ be PL functions, simplexwise linear on $K$. Consider the straight line homotopy ${f}_\lambda=(1-\lambda) f+\lambda g$, with $0\le \lambda\le 1$, and take values $\lambda_i\in [0,1]$, $1 \leq i \leq n$, as in \cref{lambda_i}. Set $\rho_i=(\lambda_i+\lambda_{i+1})/2$.

Consider the Reeb cone \eqref{cone}, with $V=X$, $R_i=X/_{\sim {f}_{\lambda_i}}$, $i=1,\ldots,n$, and $X_i=X/_{\sim {f}_{\rho_i}}$, $i=1,\ldots,n-1$.
The canonical projections $q_{\rho_i}: X \to X_i$ and $q_{\lambda_i}: X \to R_i$ are Reeb quotient maps, and the Reeb functions $R_i\to \R$ are induced by ${f}_{\lambda_i}$ as in \cref{Reeb_quotient}. We show that there are Reeb quotient maps $p_i: X/{\sim}_{{ f}_{\rho_i}}\to X/{\sim}_{{f}_{\lambda_i}}$ and $o_{i+1}: X/{\sim}_{{f}_{\rho_i}}\to X/{\sim}_{{f}_{\lambda_{i+1}}}$ that make the following diagram commute:
\[
\begin{tikzcd}[column sep={15ex,between origins},row sep=5.5mm]
 R_i = X/{\sim}_{{f}_{\lambda_i}} & & R_{i+1} = X/{\sim}_{{ f}_{\lambda_{i+1}}}
\\
 & X_i = X/{\sim}_{{f}_{\rho_i}}\ar[ul,dotted,two heads,swap,"p_i"] \ar[ur,dotted,two heads,"o_{i+1}"]&
\\
 & X\ar[u,"q_{\rho_i}",swap] \ar[ulu,"q_{\lambda_i}",bend left] \ar[uru,"q_{\lambda_{i+1}}",swap,bend right] & 
\end{tikzcd}
\]
We prove the existence of $p_i$, that of $o_{i+1}$ being analogous. 
By \cref{{lambda_i}}, there is a weakly order preserving PL surjection $\chi_i: \im f_{\rho_i} \to \im f_{\lambda_i}$ such that $\chi_i\circ f_{\rho_i}=f_{\lambda_i}$.
Hence, \cref{phi} provides the desired Reeb quotient map $p_i: X/{\sim}_{{ f}_{\rho_i}} \to X/{\sim}_{{f}_{\lambda_i}}$.

Now consider the limit $L$ over the Reeb zigzag diagram consisting of the maps $p_i$ and $o_i$, with maps $r_i: L \to X_i$ and $s_i: L \to R_i$.
Since the maps from $X$ in the above Reeb cone factor through $L$, we obtain the commutative diagram
\[
\begin{tikzcd}[column sep={15mm,between origins},row sep=5.5mm]
&& \R &  & \R & \\
\cdots&& R_i\ar[u, "\tilde f_{\lambda_i}"]& & R_{i+1}\ar[u, "\tilde f_{\lambda_{i+1}}"] && \cdots\\
& X_{i-1}\ar[ul, dotted]\ar[ur, two heads, "o_i"]& & X_i\ar[ul, two heads, swap, "p_i"]\ar[ur, two heads, "o_{i+1}"] && X_{i+1}\ar[ul,two heads, swap, "p_{i+1}"] \ar[ur, dotted]\\
&& & 
L\ar[ull, two heads, "r_{i-1}",crossing over]\ar[u,two heads, swap,"r_{i}"]\ar[urr, two heads, swap, "r_{i+1}"]
\ar[uul, two heads, end anchor=south ,"s_{i}"]\ar[uur, two heads, end anchor=south, swap, "s_{i+1}"]
 & & \\
&& & X\ar[uull,two heads, bend left,"q_{\rho_{i-1}}"]\ar[uu,two heads, bend left=45, pos=0.45, crossing over, "q_{\rho_{i}}"]\ar[uurr,two heads, bend right, "q_{\rho_{i+1}}", swap] \ar[u, "m"]& & 
\end{tikzcd}
\]
We have 
$f_{\lambda_i}
=f^L_{\lambda_i}\circ m$
for $1\leq i \leq n$, with 
$f_{\lambda_i}^L=\tilde f_{\lambda_i}\circ s_i .$
Hence, for every $\ell\in L$,
\[s^L(\ell)= \max_j f_{\lambda_j}^L(\ell)-\min_k f_{\lambda_k}^L(\ell)
\le \sum_{i=1}^{n - 1}|f^L_{\lambda_{i+1}}(\ell)-f^L_{\lambda_{i}}(\ell)|.\]
By the surjectivity of $q_{\rho_{i}}$, for every $i$ there is $x_{\ell,i}\in X$ such that $q_{\rho_{i}}(x_{\ell,i})=r_i(\ell)$. Thus,
\begin{align*}
|f_{\lambda_{i+1}}^L(\ell)-f_{\lambda_i}^L(\ell)|
&=|f_{\lambda_{i+1}}(x_{\ell,i})-f_{\lambda_i}(x_{\ell,i})|
\leq(\lambda_{i+1}-\lambda_i)\cdot \|f-g\|_\infty.
\end{align*}
In conclusion,  for every $\ell\in L$,
\[s^L(\ell)\le \sum_{i=1}^{n-1}(\lambda_{i+1}-\lambda_{i})\cdot \|f-g\|_\infty
=
\|f-g\|_\infty. \qedhere\]
\end{proof}

\begin{corollary}\label{univGraph}
$\deltaEGraph$ is a universal distance.
\end{corollary}

\begin{proof}
The claim is a direct consequence of inequality \eqref{eT<eG} together with \cref{stabilityGraph,stabilityPL,univPL}.
\end{proof}

\bibliography{Edit-biblio}

\end{document}